\documentclass[oneside,11pt]{amsart}
\usepackage[english]{babel}
\usepackage{bbm,euscript,amssymb,amsthm,amsmath}
\usepackage{mathrsfs}
\usepackage{amsfonts}
\usepackage{amsmath}
\usepackage{amssymb}
\usepackage{fancybox}
\usepackage{graphicx}
\usepackage{color}
\usepackage{lineno}
\usepackage{pstricks,pstricks-add,pst-math,pst-xkey}
\usepackage{mathabx}
\usepackage{rotating}

\usepackage{amsmath,amsfonts,amssymb,amsthm,graphicx,graphics,epsf}

\usepackage{mathcomp}
\usepackage{textcomp}

\newtheorem{thm}{Theorem} 
\newtheorem{thm*}{Theorem*}
\newtheorem{lemma}[thm]{Lemma} 
\newtheorem{theorem}[thm]{Theorem}

\newtheorem{conj}[thm]{Conjecture} 
\newtheorem{proposition}[thm]{Proposition}

\newtheorem{corollary}[thm]{Corollary}
\newtheorem{remark}[thm]{Remark} 
\newtheorem{claim}[thm]{Claim}

\def\proba{{\rm Pr}}
\def\c3n{{C_3\ \Box\ C_n}}

\def\FF{{\mathcal F}}

\def\VA{{\overline{\epsilon}}}

\def\eps{{c}}

\def\ee{{\mathcal E}}

\def\flong{{f_{\text{\rm\tiny long}}}}
\def\Flong{{F_{\text{\rm\tiny long}}}}

\def\fshort{{f_{\text{\rm\tiny short}}}}
\def\Fshort{{F_{\text{\rm\tiny short}}}}

\def\oG{{H}}
\def\oF{{F_H}}
\def\oA{{{F_H'}}}

\def\digonal{{clean}}

\def\anchor{{anchor}}

\def\Fprime{{F'}}

\def\oDelta{{\overline{\Delta}}}

\def\dd{{\mathcal D}}

\def\qu{{|E_W|}}

\def\xcoone{{\mathbf c}_1}
\def\xcoonenum{10^{-10}}
\def\xcotwo{{\mathbf c}_2}
\def\xcotwonum{10^{-5}}
\def\xcoSevennum{(10^{-5}+2)}

\def\xcothree{{\mathbf c}_3}
\def\xcofou{{\mathbf c}_4}
\def\xcofiv{{\mathbf c}_5}
\def\xcosix{{\mathbf c}_6}
\def\xcosevnobrackets{{1/2}}
\def\xcosev{{(1/2)}}
\def\xcoeig{{(1/2)(\Delta\ell+\ell^2)}}
\def\xcoeigzero{{(1/2)(\Delta_0\ell_0+\ell_0^2)}}
\def\xcoeigtwothirds{{(2/3)(\Delta\ell+\ell^2)}}
\def\xcoeigzero{{(1/2)(\Delta_0\ell_0+\ell_0^2)}}

\def\xcoSeven{{\mathbf c}_7}

\def\xcosixteen{{s}}

\def\zcoone{{\mathbf z}_1}

\def\Fwhite{{F_{\circ}}}
\def\fwhite{{f_{\circ}}}
\def\Fblack{{F_{\bullet}}}
\def\fblack{{f_{\bullet}}}

\def\somekon{{2\somelen(2\somelen+1)}}

\def\topedge{{branch}}

\def\somelen{{\ell_0}}
\def\somedeg{{\Delta_0}}
\def\sse{{(\somelen,\somedeg)}}
\def\earr{{\Xi}}

\def\ignore#1{{ }}

\def\ucr{{\hbox{\rm cr}}}

\def\real{\bbbR}
\def\bbbR{{\mathbbm R}}
\def\bbbE{{\mathbbm E}}
\def\EE{{\bbbE}}
\def\bbbX{{\mathbbm X}}

\def\ignore#1{}

\def\gg{{\mathscr G}}

\def\pp{{\mathscr P}}
\def\rr{{\mathscr R}}

\title[On the decay of crossing
  numbers of sparse graphs]{On the decay of crossing
  numbers \\ of sparse graphs}

\author{J\'ozsef Balogh}
\address{University of Illinois at Urbana-Champaign, USA.}
\email{jobal@math.uiuc.edu}
\thanks{The first author was supported by NSF CAREER Grant DMS-0745185, UIUC Campus Research Board
Grant 11067, and OTKA Grant K76099.}

\author{Jesus Lea\~nos}
\address{Unidad Acad\'emica de Matem\'aticas, UAZ. Zacatecas, Mexico.}
\email{jleanos@mate.reduaz.mx}

\author{Gelasio Salazar}
\address{Instituto de F\'\i sica, UASLP. San Luis Potosi, Mexico.}
\email{gsalazar@ifisica.uaslp.mx}
\thanks{The third author was supported by CONACYT grant 106432.}

\date{\today}



\begin{document}

\linenumbers

\begin{abstract}
Richter and Thomassen proved that every graph has an edge $e$ such
that the crossing number $\ucr(G-e)$ of $G-e$ is at least
$(2/5)\ucr(G) - O(1)$. Fox and Cs.~T\'oth proved that dense graphs have
large sets of edges (proportional in the total number of edges) whose
removal leaves a graph with crossing number proportional to the
crossing number of the original graph; this result was later
strenghtened by \v{C}ern\'{y}, Kyn\v{c}l and
G.~T\'oth. These results make our understanding
of the {decay} of crossing numbers in dense graphs essentially
complete.
In this paper we prove a similar result for large sparse graphs in which the number of edges
is not artificially inflated by operations such as edge subdivisions.
We also discuss the connection between the decay of crossing
numbers and expected crossing numbers, a concept recently introduced
by Mohar and Tamon.
\end{abstract}

\maketitle

\section{Introduction}\label{sec:intro}

\ignore{
Some important graph theoretical parameters behave very robustly under
the edge removal operation. Take, for instance, the chromatic number
$\chi(G)$ of a graph $G$. For {\em any} edge $e$ of $G$, $\chi(G-e)
\ge \chi(G) - 1$. Analogous observations hold for the genus or the
stability number of any graph. A slightly less trivial behaviour is
exhibited by parameters such as girth: it is easy to see that for all
integers $r,s$ with $s > r$, there is a graph $G$ with girth $g(G)=r$
and an edge $e$ such that $g(G) = s$. However, excluding some
uninteresting situations (such as cycles), it is
easy to show that every graph has {\em some} edge $e$ such that $g(G)
= g(G-e)$. 
}

The {\em crossing number}
$\ucr(G)$ of a graph $G$ is the minimum number of pairwise crossings
of edges in a drawing of $G$ in the plane. A graph $G$ is
$k$-{\em crossing-critical} if $\ucr(G) \ge k$, but $\ucr(G-e) < k$
for every edge $e$ of $G$.
Since
loops are totally irrelevant for crossing number purposes, all graphs under
consideration are loopless.

\subsection{The decay of crossing numbers}

In this paper we are concerned with the effect of edge removal in the
crossing number of a graph (following Fox and T\'oth~\cite{foxtoth},
this is referred to as the {\em decay of crossing
  numbers}). Richter and Thomassen~\cite{rt} proved that every graph
$G$ has some edge $e$ such that $\ucr(G-e) \ge (2/5)\ucr(G) -
37/5$. They conjectured that there always exist an edge $e$ such that
$\ucr(G-e) \ge \ucr(G) - c \sqrt{\ucr(G)}$, for some universal
constant $c$. This conjecture was proved by Fox and
T\'oth~\cite{foxtoth} for dense graphs.

Fox and T\'oth actually proved a much stronger result: the existence
of a large subset of edges whose removal leaves a graph whose crossing
number is at least a proportion of the crossing number of the original
graph. More precisely, they proved that for every fixed $\epsilon >
0$, there is a constant $n_0=n_0(\epsilon)$ such that if $G$ is a
graph with $n > n_0$ vertices and $m > n^{1+\epsilon}$ edges, then $G$
has a subgraph $G'$ with at most $(1 - \frac{\epsilon}{24})m$ edges
such that $cr(G') \ge ( \frac{1}{28} - o(1))\ucr(G)$.

This result was further strenghtened by 
\v{C}ern\'{y}, Kyn\v{c}l and
G.~T\'oth~\cite{ckt}, who proved that for every $\epsilon,\gamma>0$ there is an
$n_0=n_0(\epsilon,\gamma)$ such that if $G$ is a graph with $n > n_0$
vertices and $m> n^{1+\epsilon}$ edges, then $G$ has a subgraph $G'$
with at most $(1 - \frac{\epsilon\gamma}{1224})m$ edges
such that $cr(G') \ge (1 - \gamma)\ucr(G)$.

\subsection{The decay of crossing numbers of sparse graphs}

Due to the Fox-T\'oth and the \v{C}ern\'{y}-Kyn\v{c}l-T\'oth results, 
our understanding of the decay of
crossing numbers of dense graphs is essentially complete.  The situation
for sparse graphs is quite different. Although the Richter and
Thomassen result is fully general, it only guarantees the existence of
a single edge whose deletion leaves a graph with crossing number
substantially large. As pointed out in~\cite{foxtoth}, by combining the
following two facts one obtains an improvement to the
Richter-Thomassen result for graphs with $n$ vertices and $m > 8.1n$
edges:
(i) every graph with $m \ge \frac{103}{16}n$ satisfies $\ucr(G) \ge
0.032\frac{m^3}{n^2}$~\cite{prtt}; and (ii) for any graph $G$ and any
edge $e$ of $G$, $\ucr(G-e) \ge \ucr(G) -m + 1$~\cite{pt}. 

In this paper we investigate the decay of crossing numbers of sparse
graphs. We are particularly interested in establishing results as
similar as possible as those in~\cite{foxtoth} and~\cite{ckt}: 
the existence of large sets of edges 
whose removal leaves a graph whose crossing number is at least some
(constant) fraction of the crossing number of the original
graph. 

In contrast with dense graphs, in a sparse graph it is possible to
artificially increase the number of edges of a graph, while
maintaining its crossing number, without adding any substantial
topological feature. Consider, for instance, a graph consisting of a
large planar grid plus an additional edge $e$ joining two vertices far
apart; subdivide this additional edge $r$ times (for some integer
$r>0$) to get a path $P$, and let $G$ denote the resulting graph. For
any given $\alpha >0$, we can make $r$ sufficiently large so that any
set of at least $\alpha|E(G)|$ edges of $G$ contains at least an edge
of $P$. That is, for any set $E_0$ of at least $\alpha|E(G)|$ edges of $G$,
the crossing number of $G-E_0$ is $0$. 

This example shows that no general result can possibly be established
if we allow the number of edges to be artificially inflated. In
particular, degree $2$ vertices need to be precluded from the graphs
under consideration. This is a particular instance of a more general
way to spuriously increase the number of edges, by substituting a set
of (possibly just one) edges joining the same two vertices by a plane
connected graph, as we now describe.

We first recall the definition of a bridge.
Let $G$ be a graph, and let $u,v$ be distinct vertices of
$G$. Following Tutte, a $uv$-{\em bridge} is either a single edge
joining $u$ and $v$, together with $u$ and $v$ (in which case it is {\em trivial}), or a subgraph
of $G$ obtained by adding to a connected component $K$ of
$G\setminus\{u,v\}$ all the edges attaching $K$ to $u$ or $v$,
together with their ends.
A $uv$-bridge is {\em $uv$-planar} if it can be embedded in the
plane with $u$ and $v$ in the same face. 

Suppose that $u,v$ are distinct vertices incident with the
same face in a connected plane graph $H$ with $|V(H)| > 2$, and let $k$ be
the maximum number of pairwise edge-disjoint
$uv$-paths in $H$. We say that $(H,u,v)$ is a {\em $uv$-blob of width
  $k$}. 
Now consider a graph $G$, and let $u,v$ be vertices of $G$, joined by
$k\ge 1$ edges.
It is easy to see that we may
substitute the edges joining $u$ and $v$ by an
arbitrarily large $uv$-blob of width $k$, leaving the crossing number
(and the criticality of $G$, if $G$ is critical) unchanged.
Conversely, if $G$ is a graph with a vertex cut $\{u,v\}$, and for some
$\{u,v\}$-bridge $H$ we have that $(H,u,v)$ is a $uv$-blob of width $k$, then
 $G$ may be simplified, leaving its crossing number (and its criticality, if $G$
is critical) unchanged, by substituting $H$ 
by $k$ parallel $uv$-edges.

Note that the concept of $uv$-blob captures, in particular, the
operation of edge subdivision. 
Indeed, a subdivided edge is simply a $uv$-blob of width $1$, all of
whose vertices, other than $u$ and $v$, have degree $2$. 

\subsection{The main result}

Since we are interested in proving the existence of large sets of
edges (linear in the crossing number) with a special property (their removal
does not decrease the crossing number arbitrarily), 
we need to preclude the existence of $\{u,v\}$-bridges (for any pair
$u,v$ of vertices) that are 
$uv$-blobs, since they 
inflate the number of edges of a graph, while adding no topologically
interesting structure whatsoever to the graph itself.

As it happens, such objects are the {\em only}\ structure that needs to
be avoided. A graph is {\em irreducible} if there do not exist
vertices $u,v$ and a $\{u,v\}$-bridge $H$ such that $(H,u,v)$ is a
$uv$-blob. We prove that if $G$ is irreducible,
then a large set of its edges (linear in the crossing number) may be
removed, and still leave a graph whose crossing number is at least a
fraction of the crossing number of the original graph. More precisely:


\begin{theorem}\label{thm:main2}
For each $\epsilon>0$ and each positive integer $k$ there exist
$m_0:=m_0(\epsilon,k)$ and $\gamma:=\gamma(\epsilon)$ with the following property.
Every $2$-connected irreducible graph $G$ with $\ucr(G) = k$ and at least $m_0$
edges has a set $E_0$ of at least $\gamma
k$ edges such that $\ucr(G-E_0) > (\xcosevnobrackets - \epsilon)\ucr(G)$. 
\end{theorem}

Trivially, $3$-connected graphs are irreducible, so in particular
Theorem~\ref{thm:main2} applies to all $3$-connected graphs.

We  also apply our techniques to improve (for sufficiently large
graphs) the Richter and
Thomassen result on crossing-critical graphs. Richter and Thomassen
proved in~\cite{rt} that every graph $G$ has an edge $e$ such that
$\ucr(G-e)\ge (2/5)\ucr(G) - 37/5$.  

In order to improve on this result, again we need to be careful not to
allow the artificial inflation in the number of edges. However, we do not
need the full condition of irreducibility: it suffices to require
that 
each vertex is adjacent to at least $3$ other vertices.
A slight variant of this requirement (namely {\em $X$-minimality}) was introduced by Ding, Oporowski, Thomas, and 
Vertigan~in~\cite{dotv}, with the same motivation
of not allowing a graph with given crossing number (in their case, a
$2$-crossing-critical graph) to spuriously grow
its number of edges.


\begin{theorem}\label{thm:main3}
For each positive integer $k$, there is an integer $m_1:=m_1(k)$ with
the following property. Let $G$ be a $2$-connected graph in which each
vertex is adjacent to at least $3$ vertices.  If
$\ucr(G)=k$ and $G$ has at least $m_1$ edges, then $G$ has an edge $e$
such that $\ucr(G-e) > (2/3) \ucr(G) - 10^8$.
\end{theorem}

\ignore{
We observe that this implies an improvement on the (proportionality
constant of the) Richter and Thomassen result for all simple
graphs. Indeed, it is easy to see that a graph that 
an easy exercise to show that a graph that does not satisfy condition
(ii) of $X$-minimality has an edge whose removal decreases the
crossing number by a factor of at most one half. Thus we obtain the following
corollary even if the $X$-minimality condition is omitted.
\begin{corollary}\label{cor:main4}
For each positive integer $k$, there is an integer $m_1:=m_1(k)$ with
the following property. If $G$ is a $2$-connected simple graph
with $\ucr(G)=k$ and at least $m_1$ edges, then $G$ has an edge $e$
such that $\ucr(G-e) \ge (1/2)\ucr(G) - 10^8$.
\end{corollary}
}

We conclude this section with a brief overview of the proofs of
Theorems~\ref{thm:main2} and~\ref{thm:main3}, and of the rest of this
paper.

As in~\cite{ckt},~\cite{foxtoth}, and~\cite{rt}, we make essential use
of the
embedding method. This technique consists of finding a set $E_0$
of edges in a graph $G$, and for each $e=uv\in E_0$ a set of pairwise
edge-disjoint $uv$-paths $\pp(e)$, with the
aim of drawing $G-E_0$ (with $\ucr(G-E_0)$ crossings) and then
embedding each $e\in E_0$ very closely to some path in $\pp(e)$. The
idea is to choose the set $E_0$ so that the embedding can be done
without adding too many crossings. 

Richter and Thomassen proved the existence of an edge $e=uv$ (so that
$E_0 = \{e\}$) with the property that there is a $uv$-path (that
avoids $e$) of length at
most $4$, all of whose internal vertices have degree less than
$12$. Fox and T\'oth, and 
\v{C}ern\'{y}-Kyn\v{c}l-T\'oth used the density of $G$ to show the
existence of a large set $E_0$ of edges, such that each edge $e=uv$ of
$E_0$ has a large collection $\pp(e)$ of short edge-disjoint paths,
and such that the collections $\pp(e)$ are pairwise edge-disjoint. 

In
our current setup (sparse graphs) for all we know the graphs under
consideration may have maximum degree $3$, and so in general we cannot
expect to find collections $\pp(e)$ of more than two edge-disjoint
paths, for each $e\in E_0$.  We prove that, indeed, each graph
under consideration has 
large set $E_0$ of edges such that each $e=uv\in E_0$ has two
short $uv$-paths $P(e), Q(e)$ whose internal vertices have bounded
degree, and if $e\neq f$ then $P(e)\cup Q(e)$ and 
$P(f) \cup Q(f)$ are edge-disjoint. As it happens, $P(e)$
and $Q(e)$ are not necessarily edge-disjoint, but this turns out to be
unimportant. To be slightly more precise, let us mention that each graph $\Xi=e\cup P(e) \cup
Q(e)$ has the property that $P(e)$ and $Q(e)$ have length at most
$\ell$, and the degree of their internal vertices is less than
$\Delta$. Following the lively notation in~\cite{ckt}, we call each
$\Xi$ an $(\ell,\Delta)$-{\em earring}.

Most of the rest of this paper is devoted to proving the result
described in the previous paragraph. We start by establishing, in
Section~\ref{sec:assorted}, several assorted statements on planar
graphs; these are, in one way or another, elementary consequences of
Euler's formula.  The existence of a large set of edge-disjoint
$(\ell,\Delta)$-earrings (for certain values of $\ell$ and $\Delta$)
is proved in Section~\ref{sec:earrings} for planar graphs, and in
Section~\ref{sec:earrings2} for irreducible nonplanar graphs. 

In Section~\ref{sec:emb} we establish the version of the embedding
method that we need.  The proofs of Theorems~\ref{thm:main2}
and~\ref{thm:main3} are in
Section~\ref{sec:proofmain}.  

In Section~\ref{sec:bd} we discuss the
connection between the decay of crossing numbers and the concept,
recently introduced by Mohar and Tamon~\cite{mohartamon}, of expected
crossing numbers.
Finally, in Section~\ref{sec:conrem} we
present some concluding remarks and open questions.

\section{Assorted lemmas on planar graphs}\label{sec:assorted}

A {\em \topedge} in a graph is a path whose endpoints have degree at
least $3$, and all whose internal vertices have degree $2$. 

\begin{lemma}\label{lem:alem}
Let $G=(V,E)$ be a planar graph with minimum degree at least $2$, and
let $B\subseteq V$ be a set of vertices of degree at least $3$.
Suppose that
the number of {\topedge}s with both endpoints in $B$ is at most
$\xcosixteen$. Then there are at least $|V|/2 -
\xcosixteen/2 - (3/2)|B|$ edges with both endpoints in $V\setminus B$.
\end{lemma}

\begin{proof}
Let $W:=V\setminus B$.  To help comprehension, we color white
(respectively, black) the vertices in $W$ (respectively, $B$). A
{\topedge} is {\em black} if its endpoints are both black. A
white vertex is {\em black-covered} if all its adjacent vertices are
black. A black-covered vertex is {\em of Type I} if it has degree
$2$; otherwise (that is, if it has degree $\ge 3$) it is {\em of
  Type II}. 

Since there are no black vertices of degree $2$, then 
no black {\topedge} can contain more than one Type
I vertex. Thus there
are at most $\xcosixteen$ Type I vertices.

Let $W'$ denote the set of black-covered vertices of Type II, and
let $G'$ denote the subgraph of $G$ induced by the edges incident with
a vertex in $W'$. This is a bipartite graph with bipartition $(W'
,B')$, for
some $B'\subseteq B$. A standard  Euler formula argument yields that
$|E(G')| \le 2|V(G')| - 4 = 2|W'| + 2|B'| - 4$. Since each vertex in $W'$ has degree at least
$3$ (in $G'$, as well as in $G$) it follows that $|E(G')| = \sum_{v\in
  W'} d(v) \ge 3 |W'|$. Thus $3|W'|  \le 2|W'| + 2|B'| - 4 \le 
2|W'| + 2|B| - 4$, and so $|W'| \le 2|B| - 4$. Thus, there are at most
$2|B| - 4$ Type II vertices.

Therefore, the total number of black-covered vertices is at most 
$\xcosixteen + 2|B| - 4$. It follows that there are at least 
$|W| - \xcosixteen - 2|B| + 4 > |W| - \xcosixteen - 2|B|$ white vertices that are adjacent to at least one
white vertex, and so there are at least 
$|W|/2 - \xcosixteen/2 - |B|= |V|/2 - \xcosixteen/2 - (3/2)|B|$ 
edges with both endpoints in $W$.
\end{proof}

The {\em length} of a face in a plane graph is the length
of its boundary walk. 

A {\em digon} in an embedded graph consists of two parallel edges,
together with their common endpoints. If the endpoints are $u$ and
$v$, then it is a $uv$-{\em digon}.
A plane embedding of a graph
$G$ is {\em \digonal} if for each pair of vertices $u,v$ joined by parallel
edges, there exist edges $e,e'$ with endpoints $u$ and $v$, such that the disc bounded by the
digon formed by $e$ and $e'$ contains all edges parallel to $e$ and
$e'$, and no other edges. 

\begin{lemma}\label{cor:faces}
Let $G$ be a connected plane graph in which each vertex is adjacent to
at least $3$ vertices. Suppose that the embedding of $G$ is {\digonal}.
Let $r \ge 0$ be an integer. 
Let $F$ be the set of faces of $G$, and let $F'$ be the set of those
faces whose length is at most $r+5$.
Then $|F'|\geq \frac{r|F|+12}{r+3}$.
\end{lemma}

\begin{proof}
Let $\oG$ be a graph obtained from $G$ as follows: for each pair 
$(u,v)$ of vertices joined by parallel edges, contract to a single
all the parallel edges between $u$ and $v$.
Let $\oF$ denote the set of faces of
$\oG$, and let $\oA$ denote the set of faces of $\oG$ with length at
most $r+5$. Our first task is to show that
$|\oA| \geq \frac{r|\oF|+12}{r+3}$.

For each $f\in \oF$ the sum $w(f):=\sum_{v\sim f}1/d(v)$
  is the {\it weight} of $f$, where $d(v)$ denotes the degree of the
  vertex $v$ and $v\sim f$ means that $v$ is incident with $f$. (A
  vertex $v$ contributes to $w(f)$ as many times as the boundary walk of
  $f$ passes through $v$.) Since $H$ is simple and has minimum degree
  at least $3$,
  then, letting $l(f)$ denote the length of $f$,  we have $l(f)\geq 3$ and
  $w(f)\leq l(f)/3$.
It is easy to see that $|V(H)| = \sum_{f\in \oF}w(f)$ and
$2|E(H)|=\sum_{f\in \oF}l(f)$. 
From the last two equations and Euler's formula it follows that 
$2=\frac{1}{2} \sum_{f\in \oF}\{2w(f)-l(f)+2\}$. 

Since $w(f)\leq l(f)/3$, we have

$$12\leq \sum_{f\in \oF}\{-l(f)+6\}=\sum_{f\in \oA}\{-l(f)+6\}+\sum_{f\in
  \oF-\oA}\{-l(f)+6\}.$$ 

Since $l(f) \ge 3$ for each $f \in \oF$, then $-l(f)+6\leq 3$ and thus
$\sum_{f\in \oA}\{-l(f)+6\}\leq 3|\oA|$. If $f \in \oF-\oA$ then
$l(f)-6\geq r$, that is, $-l(f)+6\leq -r$, and so $\sum_{f\in
  \oF-\oA}\{-l(f)+6\}\leq -r(|\oF|-|\oA|)$. Thus, $12\leq
3|\oA|-r(|\oF|-|\oA|)$, and so 
$|\oA|\geq \frac{r|\oF|+12}{r+3}$, as required. 

Now as we inflate back $\oG$ to $G$, each face in $\oA$ becomes a face
in $F'$. The other faces in $F'$ are precisely the $t:=|E(G)\setminus
E(\oG)|$ faces created in the inflation process, that is, those
bounded by parallel edges.
Thus
$|F| = |\oF| + t$ and
$|F'| = |\oA| + t$. Thus 
$|F'|-t \geq \frac{r(|F|-t)+12}{r+3}$, and so
$|F'| \geq \frac{r|F|+12}{r+3} + t(1-\frac{r}{r+3}) \ge \frac{r|F|+12}{r+3}$.
\end{proof}

If $D$ is a digon in a plane graph, then 
the open (respectively, closed) disc bounded by $D$
will be denoted $\Delta(D)$ (respectively, $\oDelta(D)$). 
If $D, D'$ are digons, then we write $D'\preceq D$ 
if $\Delta(D') \subseteq \Delta(D)$. We recall that a vertex of degree
$0$ is an {\em isolated vertex}.

\begin{proposition}\label{pro:thecore}
Let $G=(V,E)$ be a plane graph, and let $Z$ be a set of isolated
vertices of $G$.
Suppose that for each digon $D$ in $G$, the 
  disc bounded by $D$ contains at least one vertex in $Z$. Then $G$
  has at most $3|V\setminus Z| + |Z|$ edges.
\end{proposition}

\begin{proof}
Let $Y:=V\setminus Z$. 
To help comprehension, we colour the vertices in $Y$ and $Z$ black and
green, respectively.

We prove the stronger statement that $G$ has at most 
$3|Y| + |Z| - 6$ edges. 
We proceed by induction on the number of digons
in $G$. In the base case $G$ has no digons, and so by Euler's Formula
it has at most $3|Y| - 6$ edges, as required.
For the inductive step, we assume that $G$ has at least one digon, and let
$D$ be a $\preceq$-minimal digon in $G$.

Suppose first that $D$
is also $\preceq$-maximal. Then let $G'$ be the graph obtained from $G$ by removing
one edge of $D$ and one green vertex contained in $\Delta(D)$. Now
$G'$ contains one fewer edge and one fewer green vertex than $G$. It is
easy to see that the induction hypothesis can be applied to $G'$, and
so the inductive step follows. 

Therefore we may assume that $D$ is not $\preceq$-maximal. Among all
digons that contain $D$, let $D'$ be a $\preceq$-minimal one.

Suppose that $D$ and $D'$ have an edge $e$ in common, and let $\overline{e}$ be
the other edge of $D$. It
is easy to see that the induction hypothesis can be applied to the
graph obtained from $G$ by removing $\overline{e}$ and a green vertex
contained in $\Delta(D)$, and once again the inductive step follows. 
Thus we may assume that $D$ and $D'$ do not have an edge in common. 

If $\Delta(D')$ contains a green vertex not contained in
$\Delta(D)$, the situation is again straightforward: the induction
hypothesis can be applied to the graph $G'$ obtained by removing one
edge of $D$ and one green vertex contained in $\Delta(D)$, and the
inductive step follows.  Thus we may assume that every green vertex
contained in $\Delta(D')$ is contained in $\Delta(D)$. 

In this case, there are no digons other than $D'$ and $D$ contained in
$\oDelta(D')$.  Now let $G'$ be the graph obtained by removing from
$G$ the black vertices and all the edges contained in
$\Delta(D')$. Let $Y'$ and $Z'$ denote the sets of black and green
vertices of $G'$, respectively, and let $E'$ denote the set of edges
of $G'$ (note that $Z'=Z$).  We may clearly apply the induction
hypothesis to $G'$, obtaining that $|E'| \le 3|Y'| + |Z| - 6$.  Let
$Y'':= Y\setminus Y'$, and $E'':= E\setminus E'$.  Let $x,y$ be the
vertices of $D'$.  Consider the graph $G''$ that consists of the
vertices in $Y''\cup \{x,y\}$ and the edges in $E''$. Since $G''$ has
exactly one digon (namely $D$), the usual Euler formula
argument yields $|E(G'')| \le 3|V(G'')| - 5$. However, this inequality
is tight only if $G''$ is maximally planar, that is, if no edge can be
added between two nonadjacent vertices while maintaining
planarity; thus, since $x$ and $y$ are not adjacent in $G''$, it follows
that $|E(G'')| \le 3|V(G'')| - 6$. Thus $|E''| \le 3(|Y''| + 2) -
6$. That is, $|E| - |E'| \le 3(|Y| - |Y'|+2) - 6$, and so $|E| \le
3|Y| + |Z|-6$, as required.
\end{proof}

A set $Z$ of vertices in a $2$-connected planar graph $G$ is an {\em \anchor} if the
following hold:

\begin{enumerate}
\item no vertex in $Z$ is part of a $2$-vertex-cut in $G$; and
\item if $\{u,v\}$ is a $2$-vertex-cut in $G$, then every nontrivial
  $uv$-bridge contains a vertex in $Z$.
\end{enumerate}

\begin{lemma}\label{lem:aux1}
Let $G$ be a $2$-connected plane graph in which each vertex is
adjacent to at least $3$ distinct vertices, and 
let $Z$ be an {\anchor} of $G$. Let $Y\subseteq V(G)\setminus Z$, and let $E_Y$ denote the set of edges of $G$ with both endpoints
in $Y$.  
Then the number
of faces of $G$ that are incident with exactly $2$ vertices of $Y$ is
at most $3|Y| + |Z| + |E_Y|$.
\end{lemma}

\begin{proof}
We may assume that $|Y| \ge 2$, as otherwise there is nothing to
prove.  Let $F_2$ denote the set of faces of $G$ that are incident
with exactly two vertices of $Y$.

We start by coloring {red} each edge in $E_Y$, and {green} each vertex
in $Z$.  Now for each $f\in F_2$, join the two vertices in $Y$
incident with $f$ by a simple {\em blue} arc contained (except, obviously, for
its endpoints) in $f$. Let $H$
denote the plane graph that consists of the vertices in $Y$ plus all
the red edges and the blue arcs (now seen as edges), as well as the
set $Z$ of green vertices. Note that the green vertices are isolated
in $H$. We remark that $|F_2|$ is the number of blue edges in $H$.

Note that if $D$ is a blue digon in $H$ (that is, both edges of $D$ are blue), with vertices
$u$ and $v$, then $\overline{\Delta}(D)$ contains a $uv$-bridge in $G$.
This bridge may be trivial (in which case it is a red edge)
or nontrivial (in which case, by hypothesis, $\Delta^o(D)$ contains a
green vertex).

Finally, let $K$ denote the graph that results from $H$ by substituting each red
edge by an isolated red vertex (placed in the interior of the red
edge). Note that $|E(K)| = |F_2|$, that the vertex set of $K$ is the union of $Y$ with the
set of all green or red vertices, and that there are $|Z|$ green and
$|E_Y|$ red vertices. 

The graph $K$ has the property that for each (necessarily blue) digon
$D$ in $K$, $\Delta^o(D)$ contains either a green or a red vertex. 
Applying Proposition~\ref{pro:thecore} we obtain that $|E(K)| \le
3|Y| + |Z| + |E_Y|$. Thus  
$|F_2| \le 3|Y| + |Z| + |E_Y|$, as required.
\end{proof}

If $G$ is a plane graph, then we let $G^o$ denote its dual.

\begin{lemma}\label{lem:aux4}
Let $G$ be a $2$-connected plane graph, 
and let $Z$ be an {\anchor}
of $G$. Suppose that the embedding of $G$ is {\digonal}. Let $\Fprime$
be a set of faces of $G$ of length at least $3$. 
Then the number of {\topedge}s in $G^o$ with both endpoints in $\Fprime$
is at most $3|\Fprime| + |Z|$. 
\end{lemma}

\begin{proof}
Since the embedding is {\digonal}, 
we may as well assume (in the context of this lemma) that $G$ has no parallel edges. It follows
that all {\topedge}s with both endpoints in $\Fprime$ are
actual edges in $G^o$.  Thus our goal is to show that there are at most
$3|\Fprime| + |Z|$ edges in $G^o$ with both endpoints in $\Fprime$.

Regarding $G$ and $G^o$ as simultaneously embedded, remove everything
except for $\Fprime$ (seen as a set of vertices in $G^o$), the edges
(in $G^o$) joining two vertices in $\Fprime$, and the vertices in $Z$.
The result is a graph $G'$ in which each vertex in $Z$ is isolated,
and such that the disc bounded by every digon contains a vertex in
$Z$. To see this last property, note that if $e$ and $f$ are the edges of a digon in
$G'$, then the edges in $G$ corresponding to $e$ and $f$ are a
$2$-edge-cut in $G$; since $Z$ is an {\anchor} set of $G$, it then
follows that the disc bounded by the digon must contain
a vertex of $Z$ in its interior.

Applying Proposition~\ref{pro:thecore}, we obtain
that $G'$ has at most $3|\Fprime| + |Z|$ edges. This finishes the
proof, since there is a bijection between the edges in $G'$ and the
edges in $G^o$ with both endpoints in $\Fprime$. 
\end{proof}

\section{Earrings in planar graphs}\label{sec:earrings}

\v{C}ern\'{y}, Kyn\v{c}l and
T\'oth introduced the lively terminology of {\em earring of size $p$} to describe a graph
consisting of an edge $e=uv$ plus a collection of $p$ pairwise
edge-disjoint, bounded-length $uv$-paths. In order to use the re-embedding
method, the goal is to find many pairwise edge-disjoint earrings.

As we mentioned in Section~\ref{sec:intro}, in our current context of
sparse graphs, where (for all we know) the graphs under consideration
may have maximum degree $3$, the best we could hope for is to prove
the existence of a large collection of earrings, each of size $2$. As
we also mentioned, in this discussion we do not need the two $uv$-paths of each earring to
be edge-disjoint, but only a weaker condition (see (iii) in the following
definition).

Let $\ell, \Delta$ be positive integers. An
$(\ell,\Delta)$-{\em earring} of a graph $G$ is a subgraph of $G$ that
consists of a {\em base} edge $e=uv$ plus two distinct $uv$-paths $P,Q$ (disjoint from $e$)
with the following properties: (i) each of $P$ and $Q$ has at most
$\ell$ edges; (ii) each internal vertex of $P$ or $Q$ has degree less
than $\Delta$; and (iii) if $f$ is an edge in both $P$ and $Q$,
then $\{e,f\}$ is a $2$-edge-cut of $G$.

An edge $e=uv$ in a $2$-connected plane graph is an $(\ell,\Delta)$-{\em edge} if each
of its two incident faces has length at most $\ell+1$, and no vertex 
incident with these two faces, other than possibly $u$ or $v$, has degree
$\Delta$ or greater. If $e$ is an $(\ell,\Delta)$-edge, then the
subgraph that consists of $e$ plus the cycles that
bound its two incident faces, is an $(\ell,\Delta)$-earring, the
$(\ell,\Delta)$-earring $\earr(e)$ {\em associated to $e$}.

The following lemma is the main workhorse in this paper.

\begin{lemma}\label{lem:workhorse}
Let $G=(V,E)$ be a $2$-connected planar graph in which each vertex is
adjacent to at least $3$ other vertices. Let $Z$ be an 
{\anchor} of $G$, where each vertex in $Z$ has degree $4$. 
Then $G$ has at least 
$\xcoonenum|E|- \xcotwonum |Z|$
pairwise edge-disjoint $(5000,500)$-earrings. 
\end{lemma}

\begin{proof} 
Throughout the proof, we make use of several constants that are either very small,
very close to $1$, or somewhat large. In order to simplify the whole
discussion, we first proceed to introduce these constants. We let
$\somelen=5000$,
$\somedeg=500$,
$\xcoone= 10^{-10}$,
$\xcotwo= 10^{-5}$,
$\xcothree= 999/1000$,
$\xcofou= 1/1000$,
$\xcofiv= 999$,
$\xcosix= 36/5000$, and
$\zcoone=3(10^{-10})$.

It is a trivial observation that every planar graph has a {\digonal}
plane embedding ({\digonal} embeddings are defined  before
Lemma~\ref{cor:faces}). 
Throughout the proof we consider a fixed {\digonal} embedding of $G$ in the plane.
Let $F$ denote the set of all faces of $G$, and let $t:=|Z|$.

\begin{claim}\label{cla:suf0}
It suffices to show that 
there are at least
$(\somekon+1)\cdot(\zcoone|F|-\xcotwo t)$ $(\somelen,\somedeg)$-edges.
\end{claim}

\begin{proof} 
Consider
the graph $H$ whose vertices are the $\sse$-edges of $G$,
with two distinct $\sse$-edges $e,f$ adjacent if $\earr(e)$ and $\earr(f)$
have some edge in common. 

We note that $H$ has maximum degree at most $\somekon$. 
This follows at once from the following two easy observations: 
(i) for each $\sse$-edge $e$, $\earr(e)$ has at most
$2\somelen$ edges other than $e$; and (ii) each edge of $G$ belongs to
at most $2\somelen + 1$ $\sse$-earrings of the form $\earr(f)$ for some edge $f$. 

Thus, $V(H)$ has a
stable set of size at least $|V(H)|/(\somekon+1)$.
Suppose that $G$ has  
at least $(\somekon+1)\cdot(\zcoone|F|-\xcotwo t)$ $\sse$-edges;
that is, 
$|V(H)| \ge (\somekon+1)\cdot(\zcoone|F|-\xcotwo t)$.
Then $H$ has a stable
set $S$ of size at least $\zcoone|F|-\xcotwo t$; that is,
there is a collection of at least 
$\zcoone|F|-\xcotwo t$ pairwise edge-disjoint 
$(\somelen,\somedeg)$-earrings.

Since $G$ has minimum
degree at least $3$, a routine Euler formula argument yields that $|F|
\ge |E|/3 + 2$. Thus there are at least 
$\zcoone(|E|/3 + 2) - \xcotwo t > \xcoone |E| - \xcotwo t$ 
pairwise edge-disjoint 
$(\somelen,\somedeg)$-earrings, as required in
Lemma~\ref{lem:workhorse}. 
\end{proof}

Let $W$ be the set of those vertices of $G$ with degree at least
$\somedeg$, and let $F_W$ denote the set of faces of $G$ that are
incident with some vertex in $W$. 
For each integer $j\ge 1$, let $F_j$
denote the set of those faces of $G$ incident with exactly $j$
vertices
in $W$ (and perhaps other vertices in $V\setminus W$), and let
$f_j=|F_j|$. 
Note that $F_W$ is the disjoint union $\bigcup_{i \ge 1}
{F_{i}}$.

Let $\Flong$ (respectively, $\Fshort$) denote the collection of faces
of $G$ with length greater than (respectively, at most) $\somelen+1$,
and let $\flong:=|\Flong|$ and $\fshort:=|\Fshort|$.  It follows
immediately from Lemma~\ref{cor:faces}
that
\begin{equation}\label{eq:fsho1}
\fshort \ge \xcothree|F|.
\end{equation}


Since $F$ is the disjoint union of $\Flong$ and $\Fshort$, then
$|F| = \flong + \fshort$, and so
$\fshort \ge \xcothree(\flong + \fshort)$ implies 
$\fshort \ge (\xcothree/(1-\xcothree))\flong$.
Note that  $\xcofiv=\xcothree/(1-\xcothree)$. Therefore,
\begin{equation}\label{eq:fsho2}
\fshort \ge \xcofiv \flong.
\end{equation}

We note that $\sum_{u\in W}d(u)=\sum_{i \ge 1}if_i$.
A routine application of Euler's formula yields
that $\sum_{i\ge 3} if_i \le 2 (3|W| - 6) = 6|W| - 12$. 
Since all vertices of $Z$
have degree $4$ it follows that $W \subseteq V\setminus Z$, and so we
can apply Lemma~\ref{lem:aux1}, to obtain $f_2 \le 3|W|
+ t + |E_W|$. Combining these observations we obtain

\begin{equation}\label{eq:imp1}
f_1 \ge 
\sum_{u\in W} d(u) - 12|W| - 2|E_W| - 2t  + 12.
\end{equation}

\begin{claim}\label{cla:anot2}
If $|F_W| > 24t + 24\xcofou\fshort$, then Lemma~\ref{lem:workhorse} follows.
\end{claim}

\begin{proof}
We establish four subclaims, and finally show that the proof follows easily
from them.

\vglue 0.3 cm
\noindent{\sc Subclaim A }{\em 
If $\qu > 6|W| -12 + \xcofou\fshort$, then Lemma~\ref{lem:workhorse} follows.
}
\vglue 0.3 cm

\begin{proof}
If $e_1, e_2, e_3$ are parallel edges with common
endpoints $u,v$, and $e_2$ is in the disc bounded by the digon formed
by $e_1$ and $e_3$, then $e_2$ is a {\em sheltered} edge. 
By Euler's formula, a simple graph on $|W|$ vertices has at most
$3|W|-6$ edges. Since the embedding of $G$ is {\digonal}, 
it follows that the subgraph of $G$ induced by $W$ has at least
$\qu - 2(3|W| - 6)=\qu - 6|W| + 12$ sheltered edges. 
The fact that $G$ is {\digonal} also implies that
each sheltered edge is a $\sse$-edge, and so $G$ has at least $\qu -
6|W| + 12$ $\sse$-edges.

Suppose that $\qu > 6|W| - 12 + \xcofou\fshort$. Then $G$ has at least
$\xcofou\fshort$ $\sse$-edges. Using (\ref{eq:fsho1}), it follows that
$G$ has at least $\xcothree\xcofou|F|$ $\sse$-edges.  The result now
follows from Claim~\ref{cla:suf0}, since 
$\xcothree\xcofou >
(\somekon+1)\zcoone$. 
\end{proof}

\vglue 0.3 cm
\noindent{\sc Subclaim B }{\em 
If $(1/6)(\sum_{u\in W} d(u)) \le 12|W| + 2t + 2\qu - 12$, then 
$|F_W| < 24t + 24\xcofou\fshort$ or else Lemma~\ref{lem:workhorse}
follows.}

\vglue 0.3 cm 

\begin{proof}
By Subclaim A, under the given hypothesis we may assume
that
$\sum_{u\in W} d(u) \le 72|W| + 12t + (72|W| - 144 +
12\xcofou\fshort) - 72 =  
144|W| + 
12t + 
12\xcofou\fshort 
- 216
< 144|W| + 
12t + 
12\xcofou\fshort$.

Since each vertex in $W$ has degree at least $\somedeg$, it follows
that $\somedeg|W| \le \sum_{u\in W} d(u)$. Hence, $ |W| < (12t +
12\xcofou\fshort)/(\somedeg-144)$.  On the other hand, obviously
$|F_W| \le \sum_{u\in W} d(u)$, and so $|F_W| < 144(12t +
12\xcofou\fshort)/(\somedeg-144) + 12t + 12\xcofou\fshort$.  Since
$144/(\somedeg-144) \le 1$, this implies
$|F_W| < 12t + 12\xcofou\fshort + 12t + 12\xcofou\fshort= 24t +
24\xcofou\fshort$.
\end{proof}

\vglue 0.3 cm
\noindent{\sc Subclaim C }{\em 
If $(1/6)(\sum_{u\in W} d(u)) \le \flong$, then
$|F_W| \le 6\fshort/\xcofiv$.}
\vglue 0.3 cm

\begin{proof}
Suppose that $(1/6)(\sum_{u\in W} d(u)) \le \flong$.
The obvious inequality $|F_W| \le \sum_{u\in W} d(u)$ then implies
that 
$|F_W| \le 6\cdot \flong$. 
The required inequality follows from (\ref{eq:fsho2}).
\end{proof}

\vglue 0.3 cm
\noindent{\sc Subclaim D }{\em 
If $(1/6)(\sum_{u\in W} d(u)) > 12|W| + 2t + 2\qu - 12$ and
\\ $(1/6)(\sum_{u\in W} d(u)) > \flong$, then
$|F_W| \le \xcosix \fshort$
or else Lemma~\ref{lem:workhorse}
follows.}
\vglue 0.3 cm

\begin{proof}
We show
that, under the given hypotheses, if $|F_W| > \xcosix \fshort$, then there are at least
$(\xcothree\xcosix/3)|F|$ $\sse$-edges; the subclaim then follows from
Claim~\ref{cla:suf0}, since
$(\xcothree\xcosix)/3 \ge (\somekon+1) \cdot \zcoone$.

It follows that, under the current hypotheses, 
\begin{equation}\label{eq:imp2}
\flong < (1/3)\sum_{u\in W} d(u)  - 12|W| - 2t - 2\qu + 12.
\end{equation}

Since $|F_1\setminus\Flong| \ge f_1 -
\flong$, 
using  (\ref{eq:imp1}) and (\ref{eq:imp2}) we obtain
$$
|F_1\setminus\Flong| \ge \sum_{u\in W} d(u) - 12|W| - 2\qu - 2t + 12 
- \flong > (2/3)\sum_{u\in W} d(u).
$$

Since each face in $F_1$ is (by definition) incident with exactly one
vertex in $W$, the inequality $|F_1\setminus\Flong| > (2/3)\sum_{u\in
  W} d(u)$ implies that at least $1/3$ of the edges incident with $W$ have their two incident faces in
$F_1\setminus\Flong$. Note that all such edges are
$(\somelen,\somedeg)$-edges. We conclude that there are at least $(1/3)
\sum_{u\in W} d(u)$ $(\somelen,\somedeg)$-edges incident with 
$W$. Since obviously $\sum_{u\in W} d(u) \ge |F_W|$, this implies
that there are at least $|F_W|/3$ $\sse$-edges.

Using the assumption $|F_W| > \xcosix \fshort$ and (\ref{eq:fsho1}),
it follows that there are at least
$(\xcothree\xcosix/3)|F|$
$\sse$-edges, as required.
\end{proof}

We now complete the proof of Claim~\ref{cla:anot2}.

Since the hypotheses of
Subclaims B, C, and D are 
exhaustive, it
follows from these subclaims that either 
we may assume that
$|F_W| < 24t + 24\xcofou\fshort$, 
or 
$|F_W| \le 6\fshort/\xcofiv$, or we may assume that
$|F_W| \le \xcosix\fshort$. 
Since $\max\{24\xcofou,6/\xcofiv,\xcosix \}=24\xcofou$,
it follows that 
we may assume that $|F_W| < 24t + 24\xcofou\fshort$.
\end{proof}

We now complete the proof of Lemma~\ref{lem:workhorse}.

A face is {\em white} if it is either 
in $\Fshort
\setminus F_W$ or has length exactly $2$, and 
is {\em black} otherwise.
We let $\Fwhite$ (respectively, $\Fblack$) denote the set of all white
(respetively, black) faces.
Let $\fwhite:=|\Fwhite|$, and 
$\fblack:=|\Fblack|$. 

Now consider the dual $G^o$ of $G$.  The $2$-connectivity of $G$
implies that $G^o$ is also $2$-connected. 
Let us say that an edge in
$G^o$ is {\em white} if its endpoints are both white (faces in
$G$). 

The key (and completely straightforward) observation is that
the  edge of $G$ associated to each white edge
is an $\sse$-edge.  Our final goal is to prove that
there are many white edges.

Every face in $\Fblack$ is either in $\Flong$ or in $F_W$, and so
$\fblack \le \flong + |F_W|$. Using (\ref{eq:fsho2}),
Claim~\ref{cla:anot2}, and the obvious
inequality  $\fshort \le |F|$, we obtain
\begin{equation}
\fblack \le 24t + (24\xcofou + 1/\xcofiv) |F|.
\end{equation}

By Lemma~\ref{lem:aux4}, $G^o$ has at most
$3\fblack + t$ 
{\topedge}s
with both endpoints black. Lemma~\ref{lem:alem} (applied to $G^o$) then
implies that there are at least
$|F|/2 - (
3\fblack + t)/2 - (3/2)\fblack
=
|F|/2 - 3\fblack - t/2
\ge 
(1/2 - 
3(24\xcofou + 1/\xcofiv)
)|F|
- (145/2)t
$
white edges. 

As we have observed, the edge of $G$ associated to each white edge 
is an $\sse$-edge. 
Thus there are at least
$
(1/2 - 
3(24\xcofou + 1/\xcofiv)
)
|F|
- (145/2)t$\
$\sse$-edges.
Since $ 
1/2 - 
3(24\xcofou + 1/\xcofiv)
\ge
(\somekon+1)\cdot\zcoone$ 
and 
$145/2 \le (\somekon+1)\cdot\xcotwo$, 
then
we are done by Claim~\ref{cla:suf0}.
\end{proof}

\section{Earrings in nonplanar graphs}\label{sec:earrings2}

\begin{lemma}\label{lem:earnonplanar}
Let $G=(V,E)$ be a $2$-connected irreducible graph.
Then $G$ has at least 
$\xcoonenum|E| - \xcoSevennum
\ucr(G)$ 
pairwise edge-disjoint
$(5000,500)$-earrings. 
\end{lemma}

\begin{proof}
Let $\ell_0:=5000$,  $\Delta_0:=500$, $\xcoone:=\xcoonenum$,
$\xcotwo:=\xcotwonum$, and $\xcoSeven:=\xcoSevennum$.
Let $t:=\ucr(G)$, and let
 $\dd$ be a drawing of $G$ with exactly $t$ crossings.
Let $H$ denote the plane graph that results by regarding the $t$ crossings
as degree $4$ vertices (this is the {\em crossings-to-vertices conversion}), which we colour green to help comprehension
(the other vertices of $H$, each of which corresponds to a vertex in
$G$, are coloured black). 
We claim that (i) each vertex in $H$ is adjacent to at least $3$ other
vertices; 
(ii) no green vertex is part of a $2$-vertex-cut; 
(iii) $H$ is $2$-connected; 
and (iv) the set of green vertices
is an {\anchor} set for $H$.

We start by noting that (i) follows easily from the irreducibility of
$G$, plus the observation that in any
crossing-minimal 
drawing of any graph, the two edges involved in any crossing cannot
have a common endpoint.

By way of contradiction, suppose that $u,v$ are green vertices such
that $\{u,v\}$ is a $2$-vertex-cut in $H$.  It is easy to see that
then there are exactly two $uv$-bridges. Let $B$ be any of these
$uv$-bridges, and let $H'$ denote the plane graph obtained from $H$ by
performing a Whitney switching on $B$ around $u$ and $v$. Now by
reversing the crossings-to-vertices conversion, we obtain from $H'$ a
drawing of $G$ in which the edge intersections corresponding to $u$
and $v$ are tangential, not crossings. Each of these two tangential
edge intersections may be removed with a small perturbation, yielding
a drawing of $G$ with two fewer crossings than $\dd$, contradicting
the crossing-minimality of $\dd$.  This contradiction shows that
$\{u,v\}$ cannot be a $2$-vertex-cut in $H$. A similar contradiction
is obtained from the assumption that $H$ has a $2$-vertex-cut with
exactly one green vertex (in this case one obtains a drawing of $G$
with one fewer crossing than $\dd$). This proves (ii).

The $2$-connectedness of $G$ readily implies that no black vertex can
be a cut vertex of $H$. On the other hand, a similar switching
argument as in the proof of (ii) shows that no green vertex can be a
cut vertex of $H$. This proves (iii).

Now let $u,v$ be black vertices such that $\{u,v\}$ is a
$2$-vertex-cut in $H$, and let $B$ be a nontrivial $uv$-bridge. If $B$ does not
contain any green vertex, then $(B,u,v)$ is clearly a $uv$-blob 
of $G$.  Since this contradicts the irreducibility of $G$,
(iv) follows. 

We can thus apply Lemma~\ref{lem:workhorse} to $H$, and obtain that $H$
has a collection $\ee$ of at least $\xcoone|E(H)| - \xcotwo t$ pairwise edge-disjoint
$\sse$-earrings. If any such earring contains a green vertex, then it
obviously contains at least two edges incident with a green vertex. 
Since these earrings are pairwise edge-disjoint, it
immediately follows that $\ee$ has a subcollection $\ee'$, with
$|\ee'| \ge |\ee| - 2t$ pairwise edge-disjoint $\sse$-earrings that do
not contain any green vertex.
That is, each earring in $\ee'$ is an $\sse$-earring of $G$.

Therefore, $\ee'$ is a collection at least 
$|\ee|-2t 
\ge \xcoone|E(H)| - (\xcotwo+2)t$ pairwise edge-disjoint
$\sse$-earrings in $G$. Since $|E(H)| \ge |E|$,
it follows that
$|\ee'| \ge \xcoone|E| - (\xcotwo + 2)t
= \xcoone|E| - \xcoSeven t$.
\end{proof}

\section{The embedding method: adding edges with few crossings}\label{sec:emb}

Our main goal is to show that every (sufficiently large) irreducible
graph has a large collection of edges whose removal leaves a graph
with large crossing number. The first main ingredient is the existence of a large collection of
pairwise edge-disjoint $(\ell,\Delta)$-earrings (for some fixed $\ell$
and $\Delta$); this is Lemma~\ref{lem:workhorse}.  The second main ingredient
is the {\em embedding method}, which was used under similar
circumstances by Richter and Thomassen~\cite{rt}, Fox and
T\'oth~\cite{foxtoth}, and \v{C}ern\'{y}, Kyn\v{c}l and
T\'oth~\cite{ckt} (see also~\cite{lei,sssv,ls}). We use the embedding
method to prove the following.

\begin{lemma}\label{lem:work2}
Let $G$ be a graph, 
and let $\ell, \Delta$, and $r$ be positive integers. Suppose that $G$ has
a collection of $r$ pairwise edge-disjoint $(\ell,\Delta)$-earrings.  Then $G$ has a
set $E_0$ of $r$ edges such that $\ucr(G-E_0) > \xcosev \ucr(G) -
\xcoeig r$.
\end{lemma}

\begin{proof}
Let $\earr_1, \earr_2, \ldots, \earr_r$ be a collection of pairwise
edge-disjoint $(\ell,\Delta)$-earrings in $G$. For $i=1,2,\ldots,r$,
let $e_i=u_iv_i$ be the base edge of $\earr_i$, and let $P_i, Q_i$ be
the $u_iv_i$-paths such that $\earr_i=P_i\cup Q_i \cup \{e_i\}$. We
shall show that $E_0:=\{e_1, e_2,...,e_r\}$ satisfies the required
property.

Let $t:=\ucr(G-E_0)$, and let $\dd$ be a drawing of $G-E_0$ with $t$
crossings.  The strategy is to extend $\dd$ to a drawing of $G$ by
drawing $e_i$ very close to either $P_i$ or $Q_i$, for
$i=1,2,\ldots,r$. Our aim is to show that this can be done while
adding relatively few crossings.

We analyze several types of crossings of $P_i$ and $Q_i$, for $i=1,2,\ldots,r$. A crossing in $\dd$ is
(i) {\em of Type 1 } if one edge is in $P_i$ and the other edge is in
$Q_i$, for some $i\in\{1,\ldots,r\}$; (ii) {\em of Type 2A } if one
edge is in $P_i\cup Q_i$ and the
other edge is in $P_j\cup Q_j$, for some $i\neq j$,
$i,j\in\{1,\ldots,r\}$; and (iii) {\em of
  Type 2B } if one edge is in $P_i\cup Q_i$ for some
$i\in\{1,\ldots,r\}$ and the other in $E(G)\setminus
\bigcup_{j=1}^r (P_j\cup Q_j)$. 
Note that if a crossing $\times$ involving an edge of $\bigcup_{i=1}^{r}
P_i\cup Q_i$ is neither of Type 1, nor 2A, nor 2B, then the edges
involved in  $\times$ must
be both in $P_i$ or both in $Q_i$, for some $i\in \{1,2,..,r\}$. As we shall see, this
last type of crossing is irrelevant to our discussion.

For $i=1,2,\ldots,r$ and $k\in\{1,2\}$, let $\chi_k(P_i)$ (respectively, $\chi_k(Q_i)$)
denote the number of crossings of Type $k$ that involve an edge in 
$P_i$ (respectively, $Q_i$).

In every crossing-minimal drawing of any graph, no pair of edges cross
each other more than once. Since each of $P_i$ and $Q_i$ has at most
$\ell$ edges, it follows that
\begin{equation}\label{eq:chi1}
\chi_1(P_i) \le \ell^2, \text{\rm for }
i=1,\ldots,r.
\end{equation}

Now let $\rr$ be the set of all sequences $(R_1,R_2,\ldots, R_r)$, with
$R_i\in\{P_i,Q_i\}$ for $i=1,2,\ldots,r$, and consider the sum 
$\Sigma:=\sum_{R \in \rr} 
\bigl(
\sum_{i=1}^r \chi_2(R_i) 
\bigr)$.

We claim that a crossing of Type 2A contributes in exactly
$2^{r}$ to $\Sigma$. To see this, first note that such a crossing involves an edge
of an $R_i\in \{P_i,Q_i\}$ and an edge of an $R_j\in\{P_j,Q_j\}$ for
some $i\neq j$. Let $T_i$ (respectively, $T_j$) be the element in
$\{P_i,Q_i\}\setminus R_i$ (respectively, $\{P_j,Q_j\}\setminus R_j$).
There are $2^{r-2}$ sequences in $\rr$ that include both $R_i$ and
$R_j$, and so for each such sequence, the crossing contributes in $2$
to $\Sigma$. There are $2^{r-2}$ sequences in $\rr$ that include $R_i$
and do not include $R_j$, and so for each such sequence, the crossing
contributes in $1$ to $\Sigma$.
Analogously, there are $2^{r-2}$ sequences in $\rr$ that include $R_j$
and do not include $R_i$, and so for each such sequence, the crossing
contributes in $1$ to $\Sigma$. Therefore each crossing of Type 2A
contributes  in $2\cdot 2^{r-2} + 2^{r-2} + 2^{r-2}=2^r$ to $\Sigma$,
as claimed.
Note that this reasoning assumes that no crossing
of Type 2A is in both $P_i$ and $Q_i$ for the same $i$. This is 
immediate if $P_i$ and $Q_i$ are edge-disjoint, but we recall from our
definition of earring that $P_i$ and $Q_i$ may share edges. However,
the validity of our reasoning follows since (again, by the definition of
earring) any edge $f\in E(P_i)\cap E(Q_i)$ is a cut edge of $G-e_i$,
from which it follows that $f$ cannot be crossed in any optimal drawing
of $G-E_0$. 

We also note that a crossing of Type 2B contributes
to $\Sigma$ in exactly $2^{r-1}$. Indeed, such a crossing involves
(for some fixed $i$) an
edge of $R_i$ and an edge that belongs to no $R_j$; it contributes in
$1$ to $\chi(R_i)$, and there are $2^{r-1}$ sequences in $\rr$
that include $R_i$. (As in the previous paragraph, we remark that we
are making use of the valid assumption that no crossing is in
both $P_i$ and $Q_i$ for the same $i$). 

In conclusion, each crossing of Type 2A or 2B contributes to $\Sigma$
in at most $2^r$. 
Since only crossings of Types 2A and 2B contribute to $\Sigma$, and
$\dd$ has $t$ crossings in total, we conclude that 
$\sum_{R \in \rr} 
\bigl(
\sum_{i=1}^r \chi_2(R_i) 
\bigr) 
\le 2^{r}t. 
$
Since $|\rr| = 2^r$, 
it follows that 
for some sequence
$(R_1,R_2,\ldots,R_r)\in\rr$, $\sum_{i=1}^r \chi_2(R_i) \le t$. 
By relabeling (exchanging) $P_i$ and $Q_i$ if necessary, we may assume without any
loss of generality that $R_i=P_i$ for each $i=1,2,\ldots,r$, and so 
\begin{equation}\label{eq:chi2}\
\sum_{i=1}^r \chi_2(P_i) \le t.
\end{equation}

Now note that some $P_i$ may have self-crossings. However, for each
$i$ there is a simple curve $\alpha_i$, contained in $P_i$, joining
$u_i$ and $v_i$.  The definition of crossings of types 1, 2A, and 2B
obviously extend to the crossings on each $\alpha_i$, and so
(\ref{eq:chi1}) and (\ref{eq:chi2}) imply that $\chi_1(\alpha_i) \le
\ell^2$ for $i=1,2,\ldots,r$, and $\sum_{i=1}^r \chi_2(\alpha_i) \le t$.  Moreover (this
is the effect of having obtained $\alpha_i$ by avoiding the
self-crossings of its corresponding $P_i$), for $i=1,2,\ldots,r$, each
crossing of $\alpha_i$ is of one of these types.

The idea is to draw each $e_i$ very close to its corresponding
$\alpha_i$. There are two kinds of crossings on the resulting drawings
of $e_i$, $i=1,\ldots, r$. Some crossings occur as we traverse $e_i$
and pass very close to a crossing of $\alpha_i$. The 
inequalities in the previous paragraph imply that there are, in total,
at most 
$\ell^2r + t$ 
crossings of this first kind.  The second kind of crossing
occurs as we pass very close to a vertex in $\alpha_i$, and cross some
edges incident with this vertex. Since each such vertex is an internal
vertex of some $P_i$ (that is, has degree $<\Delta$) and there are at
most $\ell-1$ internal vertices in each $P_i$, we conclude that each
$e_i$ has fewer than $\Delta\ell$ crossings of this second kind. Thus
in total there are fewer than $\Delta\ell r$ crossings of the second
kind.

We conclude that all the edges 
$e_1, e_2, \ldots, e_r$ may be added to the drawing $\dd$ of $G-E_0$ by
introducing fewer than
$(\Delta\ell + \ell^2)r + t$ crossings. Since $t=\ucr(G-E_0)$, it
follows that
$\ucr(G) <  2\ucr(G-E_0) +
(\Delta\ell + \ell^2)r$ or, equivalently,
$\ucr(G-E_0) > (1/2)\ucr(G) - (1/2)(\Delta\ell + \ell^2)r$.
\end{proof}

If we are interested in removing only one edge (as we are in
Theorem~\ref{thm:main3}), we can improve the $\xcosevnobrackets$
coefficient in Lemma~\ref{lem:work2} to $2/3$, as the following
statement shows.

\begin{lemma}\label{lem:work3}
Let $G$ be a graph, 
and let $\ell$ and $\Delta$ be positive integers. Suppose that $G$ has
an $(\ell,\Delta)$-earring.  Then $G$ has an edge $e$ 
such that $\ucr(G-e) > (2/3) \ucr(G) -
\xcoeigtwothirds$.
\end{lemma}

\begin{proof}
The proof is essentially the same as the proof of
Lemma~\ref{lem:work3}, with the following favourable exception. If
we consider only one earring, then $r=1$, and so there are no crossings of Type 2A.  Each crossing
of Type 2B
contributes to $\Sigma$ in at most $1$, and so $\chi_2(P_1) +
\chi_2(Q_1) \le t$. By exchanging
$P_1$ and $Q_1$ if necessary, we may assume that $\chi_2(P_1) \le t/2$. 

In parallel to the last paragraph of the proof of
Lemma~\ref{lem:work3}, in the present case we conclude that 
the edge $e_1$ may be added to the drawing $\dd$ of $G-E_0 =
G-e_1$ by introducing fewer than 
$(\Delta\ell + \ell^2) + t/2$ crossings. Since $t=\ucr(G-e_1)$, it
follows that
$\ucr(G) <  (3/2)\ucr(G-e_1) +
\Delta\ell + \ell^2$ or, equivalently,
$\ucr(G-e_1) > (2/3)\ucr(G) - (2/3)(\Delta\ell + \ell^2)$.
\end{proof}

\section{Proof of Theorems~\ref{thm:main2} and~\ref{thm:main3}}\label{sec:proofmain}


\begin{proof}[Proof of Theorem~\ref{thm:main2}]
Let $\ell_0:=5000$ and $\Delta_0:=500$, $\xcoone:=\xcoonenum$, and
$\xcoSeven:=\xcoSevennum$.  
Let $k$ be a positive integer and let $\epsilon > 0$.  
Define $\gamma:= \epsilon/(\xcoeigzero)$ and 
$m_0:= 
((\xcoSeven + \gamma)k)/\xcoone$.
Let $G=(V,E)$ be a $2$-connected irreducible graph with 
$\ucr(G) = k$ and at least 
$m_0$ edges.

Lemma~\ref{lem:earnonplanar} implies that $G$ has a collection
of at least
$\xcoone|E| - \xcoSeven k$ pairwise
edge-disjoint $\sse$-earrings. Since $|E| \ge ((\xcoSeven + \gamma)k)/\xcoone$, it follows
that $G$ has a collection of at least 
$\gamma k$ pairwise edge-disjoint $\sse$-earrings.
Thus, by Lemma~\ref{lem:work2}, $G$ has
a collection $E_0$ of at least $\gamma k$ edges such that
$\ucr(G-E_0) > \xcosev \ucr(G) - \xcoeigzero\gamma k 
= \xcosev \ucr(G) - \epsilon k =
(\xcosev-\epsilon) \ucr(G)$.
\end{proof}

If $u,v$ are vertices of a graph $G$, a {\em double $uv$-path} is a
subgraph of $G$ that consists of a $uv$-path with all its edges
doubled.


\begin{proof}[Proof of Theorem~\ref{thm:main3}]
Let $\ell_0:=5000$, $\Delta_0:=500$, 
$\xcoone:=\xcoonenum$, and
$\xcoSeven:=\xcoSevennum$.  Let $k$ be a positive integer, and let $m_1:=
(\xcoSeven k)/\xcoone + 1$.  We prove that if $G=(V,E)$ is a $2$-connected 
graph in which each vertex is adjacent to at least $3$ vertices,
$\ucr(G) = k$,  and $G$ has at least $m_1$ edges, then $G$ has an
edge $e$ such that $\ucr(G-e) > (2/3)\ucr(G) - 10^{8}$.

Suppose first that $G$ is not irreducible, and let $(B,u,v)$ be a
minimal blob in $G$, (that is, $G$ has no blob $(B',u',v')$ such that
$B'$ is a subgraph of $B$). The minimality of $B$ implies that $B$ has
no cut edges, and so its width $w(B)$ is at least $2$. It is easy to
see that if every edge of $B$ is in a $2$-edge-cut separating $u$ and
$v$, then $B$ is a double $uv$-path. This clearly contradicts the
$X$-minimality of $G$, and so we conclude that there is an edge $e$ in
$B$ such that the $uv$-blob (in $G-e$) $B-e$ has width at least $2$.

By way of
contradiction, suppose that $\ucr(G-e) < (2/3)\ucr(G)$. It is
straightforward to see that there is a crossing-minimal drawing
$\dd$ of $G-e$ in which the set $E'$ of edges crossed in $B-e$ form a
smallest $uv$-{\em edge cut} (that is, a minimum size edge cut in
$B-e$ separating $u$ and $v$), with each edge in $E'$ crossed the same
number (say $s$) of times. In particular, $\ucr(G-e) \ge |E'|s \ge
2s$.  The planarity of $B-e$ (with $u,v$ in the same face) implies
that: (i) if $e$ is in distinct components of $(B-e)-E'$, then $e$ can be
added to $\dd$ by introducing exactly $s$ crossings; and (ii) otherwise, $e$ can be
added to $\dd$ without introducing any crossings. In either case, the
result is a drawing of $G$ with at most $\ucr(G-e) + s$ crossings, and
so $\ucr(G) \le \ucr(G-e) + s$. The assumption $\ucr(G) >
(3/2)\ucr(G-e)$ then implies $\ucr(G-e) < 2s$, contradicting that
$\ucr(G-e) \ge 2s$. Thus $\ucr(G-e) \ge (2/3)\ucr(G) > (2/3)\ucr(G) -
10^8$.

Suppose finally that $G$ is irreducible. 
Lemma~\ref{lem:earnonplanar} then implies that $G$ has 
at least
$\xcoone|E| - \xcoSeven k$ pairwise
edge-disjoint $\sse$-earrings. Since $|E| \ge (\xcoSeven k)/\xcoone + 1$, it follows
that $G$ has at least one  $\sse$-earring.
Thus, by Lemma~\ref{lem:work3}, $G$ has
an edge $e$ such that
$\ucr(G-e) > (2/3) \ucr(G) - \xcoeigtwothirds > (2/3)\ucr(G) - 10^8$.
\end{proof}

\section{Bounded decay and expected crossing numbers}\label{sec:bd}

The pioneering work of Richter and Thomassen, as well as our work
in this paper, are naturally described as ``bounded decay'' results:
the existence of sets of edges whose removal does not decrease
arbitrarily the crossing number. The papers by Fox and T\'oth~\cite{foxtoth} and 
by \v{C}ern\'{y}, Kyn\v{c}l and T\'oth~\cite{ckt} concern themselves
with ``almost no decay'' results: the existence of sets of edges whose
removal results in a very small decrease of the crossing number. 


As an additional motivation to bounded decay results, we discuss in this
section a connection with expected
crossing numbers, a concept recently introduced by Mohar and
Tamon~\cite{mohartamon,mohartamon2}.

\subsection{Expected crossing numbers and decay of crossing numbers}

Given a drawing $\dd$ of a graph
$G=(V,E)$, and a {\em weight function} $w:E\to \real_+$, define the
{\em crossing weight} $\ucr(\dd,w)$ as $\sum_{\{e,f\} \in \bbbX(\dd)}
w(e) w(f)$, where $\bbbX(\dd)$ is the set of all pairs of edges that
cross each other in $\dd$. The pair $(G,w)$ is a {\em weighted graph},
and the {\em weighted crossing number} of $(G,w)$ is
$\ucr(G,w):=\min_{\dd} \ucr(\dd,w)$, where the minimum is taken over all
drawings $\dd$ of $G$.  Now take the weights on the edges to be
independently identically distributed random variables, with uniform
distributions on the interval $[0, 1]$. The expected value of
$\ucr(G,w)$ under this distribution is the {\em expected crossing
  number} of $G$, and is denoted $\bbbE(\ucr(G))$.

Let us say that a family $\gg$ of graphs is {\em robust} (or, more
precisely, $\epsilon$-{\em robust}) if there
exist a constant $\epsilon:=\epsilon(\gg)$ and an $n(\gg)$ such that
$\bbbE(\ucr(G)) \ge \epsilon\cdot \ucr(G)$ for every graph $G$ in
$\gg$ with at least $n(\gg)$ vertices.  

Mohar and Tamon proved in~\cite{mohartamon} that $\bbbE(\ucr(K_n))$ is
$\Theta(n^4)$. From this it follows immediately that the family of all
complete graphs is robust. Moreover, it follows from their Crossing
Lemma for Expectations (Theorem 5.2 in~\cite{mohartamon}) that for
each fixed $\gamma > 0$, the family of graphs with at least
$\gamma\cdot n^2$ edges is also robust
(more precisely, $\epsilon$-robust, where $\epsilon$ might depend on
$\gamma$). 
It is thus natural to inquire
about the robustness of families of sparser graphs.


Our aim in this subsection is to unveil and exploit the close connection between
robustness and  several results and conjectures,
presented in~\cite{ckt}, on the decay of crossing numbers.

In ~\cite{ckt}, \v{C}ern\'{y}, Kyn\v{c}l and T\'oth proved the
following: for each $\epsilon > 0$, there exist $\delta,\gamma > 0$
such that every sufficiently large graph $G$ with $n$ vertices and
$m\ge n^{1+\epsilon}$ edges has a subgraph $G'$ with at most
$(1-\delta)m$ edges such that $\ucr(G') \ge \gamma\cdot\ucr(G)$. This
impressive ``almost no decay'' statement is best possible, in the
sense that (as shown in~\cite{ckt}) one cannot require that {\em
  every} subgraph with $(1-\delta)m$ edges has crossing number at
least $\gamma\cdot\ucr(G)$.  
In this vein, \v{C}ern\'{y}, Kyn\v{c}l and
T\'oth also investigated the following closely related problem. 

Let us say that a family
$\gg$ of graphs is {\em stable} (or, more precisely,
$(\delta,\gamma)$-{\em stable}) if there exist positive constants
$\delta:=\delta(\gg)$, $\gamma:=\gamma(\gg)$, and $n(\gg)$ 
such that for every graph $G\in \gg$ with at least $n(\gg)$ vertices
(and $m$ edges), 
a positive fraction of all subgraphs of $G$ with
$(1-\delta)m$ edges has crossing number at least $\gamma\cdot\ucr(G)$. 
The requirement may be equivalently formulated as follows: if $G'$ is
a random subgraph of $G$ obtained by deleting independently each edge
with probability $\delta$, then w.h.p.~$\ucr(G') \ge
\gamma\cdot \ucr(G)$. 

In the earlier version~\cite{ckt0} of~\cite{ckt}, it was conjectured
that for each $\epsilon > 0$, the family of graphs with
$\Theta(n^{1+\epsilon})$ edges is stable. 
In~\cite{ckt}, it was shown that this is false
for $\epsilon < 1/3$ (we have slightly refined the construction
in~\cite{ckt}, and shown that it does not hold either for $\epsilon =
1/3$; see Theorem~\ref{thm:cktrefined}).  The conjecture
remains open for denser graphs:

\begin{conj}\label{con:ran}
There exists an $\VA \in (1/3,1)$ such that, for each $\epsilon\in (\VA,1]$, 
the family of graphs with $\Theta(n^{1+\epsilon})$ edges is stable. 
\end{conj}

(See also a weaker version put
forward in~\cite{ckt}).

\ignore{We note that such a statement does not hold if we substitute the
density requirement $m\ge n^{1+\epsilon}$ by a weaker bound such as
$m\ge n\cdot 
\log n/(\log{\log n})$. 
Indeed, the graph $G_n$ obtained from
$C_3\ \Box \ C_n$ by substituting each edge with $
\log n/(\log{\log n})
$ 
parallel edges has 
crossing number $n\cdot 
(\log n/(\log{\log n}))^2$, and yet if we take a subgraph $G'_n$ of $G_n$ in which each
edge is chosen independently with probability $1-\delta$ (for any
fixed $\delta>0$), then
w.h.p. $G'_n$ is planar.}


Before moving on to explore the close relationship between
Conjecture~\ref{con:ran} and the robustness of dense graphs, we note
the stability of random graphs:

\begin{remark}
The family of all random graphs $G(n,p)$ with $p> 2/n$, is stable.
\end{remark}

\begin{proof} We start by noting that $\EE(\ucr(G(n,p))
\le p^2 \ucr(K_n) \le (1/10) p^2 n^4.$
From the other side, Spencer and G.~T\'oth (\cite{spencertoth}, Section 4) proved that there
is a $c>0$ such that for $n$ sufficiently large
 the lower bound $\EE (\ucr(G(n,2/n))) > c  n^2$ holds. Standard sparsening of
$G(n,p)$ (keeping each edge with probability $2/(pn)$) gives that for
$p>2/n$, $\EE (\ucr(G(n,p))) >  (c/4)   p^2 n^4$. Using these bounds, together with the
observation that if each edge of a $G(n,p)$ is removed with probability
$\epsilon$ then we obtain a $G(n,(1-\epsilon)p)$, the remark follows.
\end{proof}

The key connection between expected crossing number (robustness) and the decay of
crossing numbers (stability) is the following observation:

\ignore{, besides its
independent interest, Conjecture~\ref{con:ran} is equivalent to the
robustness of dense graphs, as follows:}

\begin{proposition}\label{pro:theconnection}
If a family $\gg$ of graphs is stable, then it is robust. More precisely:
if $\gg$ is $(\delta,\gamma)$-stable, then it is $\delta^2\gamma$-robust.
\end{proposition}

\begin{proof}
Suppose that $\gg$ is a $(\delta,\gamma)$-stable family of
graphs. Let $G$ be a (sufficiently large) graph in $\gg$, and let $w$
be a random weight assignment (sampled from the uniform distribution) on
the edges of $G$. Our aim is to show that the expected value of
$\ucr(G,w)$ is at least $\delta^2\gamma\cdot \ucr(G)$. 

Let $G'$ be the subgraph of $G$ that results by
deleting the edges that receive a weight smaller than $\delta$ under
$w$. Let $\dd$ be a drawing of $G$ that minimizes $\ucr(G,w)$, and let
$\dd'$ be the restriction of $G$ to $G'$. Clearly $\dd'$ has at most
$\ucr(G,w)/\delta^2$ crossings, and so $\ucr(G') \le \ucr(\dd') \le
\ucr(G,w)/\delta^2$. Thus $\ucr(G,w) \ge \delta^2 \ucr(G')$. 

Note that $G'$ may be equivalently regarded as a graph obtained
from $G$ by deleting each edge independently with probability
$\delta$. Since $\gg$ is $(\delta,\gamma)$-stable, it follows that
w.h.p.~$\ucr(G') \ge \gamma\cdot\ucr(G)$. Therefore the expected value
of $\ucr(G,w)$ is at least $\delta^2\gamma\cdot\ucr(G)$, as required.
\end{proof}

We now proceed with a concrete illustration of how the
results and techniques on the decay of crossing numbers (specifically,
those developed in~\cite{ckt}) find an immediate application in
expected crossing numbers.

As we observed above, \v{C}ern\'y, Kyn\v{c}l and
T\'oth~\cite{ckt} proved that, for each $\epsilon\in (0,1/3)$, the family of
graphs with $\Theta(n^{1+\epsilon})$ edges is {\em not} stable. We have
slightly refined the construction
in~\cite{ckt}, and extended it to cover the case $\epsilon =
1/3$. 

\begin{theorem}[Non-stability of graphs with $\Theta(n^{4/3})$ edges]\label{thm:cktrefined}
For every $\delta,\gamma>0$ there exist $c:=c(\delta,\gamma)$ and
$n_0:=n_0(\delta,\gamma)$ such that there exist infinitely many graphs $G$ with $n >
n_0$ vertices and $c\cdot n^{4/3} < m < n^{4/3} $ edges, that satisfy the
following. If $G'$ is a random subgraph of $G$ obtained by deleting
independently each edge with probability $\delta$, then w.h.p.
$$
\ucr(G') < \gamma\cdot \ucr(G).
$$
\end{theorem}

We omit the proof of this result, since it closely resembles the proof
of our next statement. Theorem~\ref{thm:robbust} shows the
non-robustness of graphs with $\Theta(n^{4/3})$ edges, and illustrates
how the non-stability results and techniques in~\cite{ckt} can be extended to prove
the non-robustness of graphs with $\Theta(n^{1+\epsilon})$ edges for
each $\epsilon\in(0,1/3)$.

\begin{theorem}[Non-robustness of graphs with $\Theta(n^{4/3})$ edges]\label{thm:robbust}
For every $\gamma>0$ there exist $c:=c(\gamma)$ and
$n_0:=n_0(\gamma)$ such that there are infinitely many graphs $G$ with $n >
n_0$ vertices and $c\cdot n^{4/3} < m < n^{4/3}$ edges, and 
$$
\EE(\ucr(G)) < \gamma\cdot \ucr(G).
$$
\end{theorem}

\begin{proof} 
For readability purposes, we shall omit
explicitly taking the integer part of several quantities involved. The
integrality requirement will be, in every case, obvious from the context. 

We may assume without loss of generality that
$\gamma$ is small enough so that 
$e^{-1200/\gamma}<\gamma/720$.  Let $\alpha:=\gamma/600$,
$\eps:= \alpha^2/100$, $r:=\alpha^2
n^{1/3}/5$, $s:=1/\alpha^2$, and $t:=\sqrt{n/s}$. Note
that obviously $r > 5\eps n^{1/3}$. 



Inspired by the construction in~\cite{ckt}, $G$ will be the disjoint
union of two graphs 
$G_1$ and $G_2$ plus some isolated vertices. Let $G_1$ be $
n/2r$ copies of the complete graph $K_r$. Clearly $|V(G_1)|\le
n/2$.  Now let $G_2$ be obtained from a complete graph $K_t$ by
subdividing each edge $s-1$ times, i.e. replacing each edge by a path
with $s$ edges (these length $s$ paths are the {\em branches}).  It is easy to
check that
$|V(G_2)|\le n/2$. Furthermore,
\begin{equation}\label{eq:somm}
\alpha^4 n^2 =  t^4> \ucr(G_2)> \frac{t^4}{100}=
\frac{n^2}{100s^2}=\frac{ \alpha^4  n^2}{100},
\end{equation}
where the inequalities $t^4 > \ucr(G_2) > t^4/100$ are easily derived
bounds for the crossing number of the complete graph on $t$ vertices.

Now let $w$ be a random weight assignment on the edges of $G$. Let
$E_{<\alpha}$ denote the set of edges of $G$ that receive a weight smaller than
$\alpha$ under $w$. Let us say that a branch is {\em weak} if at
least one
of its edges is in $E_{< \alpha}$; otherwise the branch is {\em strong}. 

The probability that any fixed branch is strong is
$$(1-\alpha)^s\approx e^{-\alpha s}=e^{-1/\alpha}.$$

Using Chernoff's bound, w.h.p.~at most $t^2 e^{-1/\alpha}$ branches are
strong. That is, w.h.p. at least ${t\choose 2} - t^2 e^{-1/\alpha}\approx t^2(1/2 -
e^{-1/\alpha})$ branches are weak.

Now consider the drawing of $G_2$ in which the $t$ vertices of degree $t-1$
are in convex position, and the edges are the straight
segments joining them. This drawing of $G_2$ has ${t\choose 4}\approx
t^4/24$ crossings (this is by no means a crossing-minimal drawing of
$G_2$, but it is enough for our purposes). Moreover, by adjusting the
drawing of each branch if needed, we may ensure that each branch is
crossed in exactly one edge, namely the edge with smallest weight. It
follows that the number of crossings involving two strong branches
(and thus, in particular, the number of crossings of weight $\ge
\alpha$) is w.h.p.~at most $(t^2 e^{-1/\alpha})^2$, and so w.h.p.


\begin{align}
\nonumber
\ucr(G_2,w)&< t^4 e^{-2/\alpha} + \alpha\cdot
t^4(
1/24 - e^{-2/\alpha}
)
< t^4(
\alpha/24 + e^{-2/\alpha}
)\\
&< 100\ucr(G_2)(\alpha/24 + e^{-2/\alpha}) \le 5\alpha\cdot\ucr(G_2), \label{eq:somn}
\end{align} 
where for this last inequality we used that 
$e^{-1200/\gamma}=e^{-2/\alpha}<\gamma/720=(5/6)\alpha$.

We finally move on to $G$. 
First we note that $$|E(G)|=|E(G_1)|+|E(G_2)|\ge |E(G_2)|=  (n/2r) {r(r-1)}/{2}>nr/5> \eps n^{4/3}.$$
Using (\ref{eq:somm}), we obtain
\begin{equation}\label{eq:zom}
\ucr(G)=\ucr(G_1)+\ucr(G_2) > \ucr(G_2) >  \alpha^4 n^2/100.
\end{equation}
From the other side, using (\ref{eq:somm}) and (\ref{eq:somn}) and the trivial bound
$\ucr(K_r) \le r^4$, we get
\begin{equation}\label{eq:zum}
\ucr(G,w)\le \ucr(G_1)+\ucr(G_2,w)\le  
(n/2r) r^4 + 5\alpha^5 n^2 < 6\alpha^5n^2,
\end{equation}
where for the last inequality we used the (easily checked) inequality
$(n/2r) r^4 < \alpha^5n^2$.

Finally, using (\ref{eq:zom}) and (\ref{eq:zum}) and recalling that
$\alpha=\gamma/600$, we obtain
$$ \ucr(G,w) <   6\alpha^5  n^2 = (600\alpha)(\alpha^4 n^2/100) < 
\gamma\cdot \ucr(G),$$
as required.
\end{proof}


We close this subsection with two constructions that further
illustrate the discrepancy between the crossing number of a
graph and its expected crossing number.

First we describe a construction that highlights the fact that the
crossing number (of a family of graphs) may grow with the number of
vertices, and yet the expected crossing number (of all graphs in the
family) may be bounded by an absolute constant.
For any graph $G$, let $n(G)$ and $m(G)$ denote the number of vertices
and edges of $G$, respectively, and let $s\cdot G$ the graph that
consists of $s$ disjoint copies of $G$. Let $K_5(t)$ denote the graph
obtained by replacing each edge of $K_5$ with a path of length $t$ (a
{\em branch}).  Trivially, for any positive integer $s$, $n(s\cdot
K_5(t))=s(10(t-1)+5)=10st-5s, m(s\cdot K_5(t))=10st$, and $\ucr(s\cdot
K_5(t))=s.$ However, the weighted crossing number of $K_5(t)$ is $\min
w(e)w(f)$, where the minimum is taken over all pairs of edges $e,f$
that lie on branches that correspond to nonincident edges.  A fairly
standard calculation  shows that
$\EE(\ucr(s\cdot K_5(t))\le (s/t^2) \log^2 s$.  It is worthwhile to
explore the consequences of plugging in various values of $s$.
Probably the most interesting case occurs when $s=n^{2/3}/\log n$, for
this shows the following:

\begin{proposition}\label{pro:hignon}
There exists an infinite family of graphs $G$ with crossing number
$n^{2/3}/\log n$ and expected crossing number at most $1$. \hfill$\Box$
\end{proposition}

Our final construction pertains a family of graphs that seem more
natural than the graphs constructed above. We recall that
$\c3n$ 
denotes the Cartesian product of the cycles of sizes $3$ and $n$ (see
Figure~\ref{fig:figc3n}). 


\begin{proposition}\label{pro:lincro}
The Cartesian products $\c3n$ satisfy
$$
\ucr(\c3n) = n,
$$
and yet
$$
\EE(\ucr(\c3n)) \le 2n^{2/3} \log^{1/3} n + 3.
$$ 
\end{proposition}

\def\kaynoconstant{{{\log n}/(\log{\log n})}}
\def\invthreekaynoc{{(\log{\log n})/{3\log n})}}
\def\kay{{{\log n}/(10\log{\log n})}}
\def\invthreekay{{10\log{\log n}/{3\log n}}}

\begin{proof}
The
vertices of $\c3n$ can be labeled $v_{i,j}$, $0\le i \le 2$, $0 \le j
\le n-1$, so that there is an edge joining $v_{i,j}$ and $v_{i',j'}$ if
and only if either (i) $j=j'$ and $|i-i'|=1$ or 
(ii) $i=i'$ and $|j-j'|=1$ (indices are modulo $n$). For $j=0,1,\ldots,n-1$, let
$V_j:=\{
v_{i,j}\ | \ i\in \{0,1,2\}
\}$. That is, the $V_j$s are the vertex sets of the $3$-cycles. For
$j=0,1,\ldots,n-1$,  let $E(j)$ denote the set of (three) edges with
an endpoint in $V_j$ and another endpoint in $V_{j+1}$. 

\begin{center}
\begin{figure}[htb]
 	\scalebox{0.6}{\input{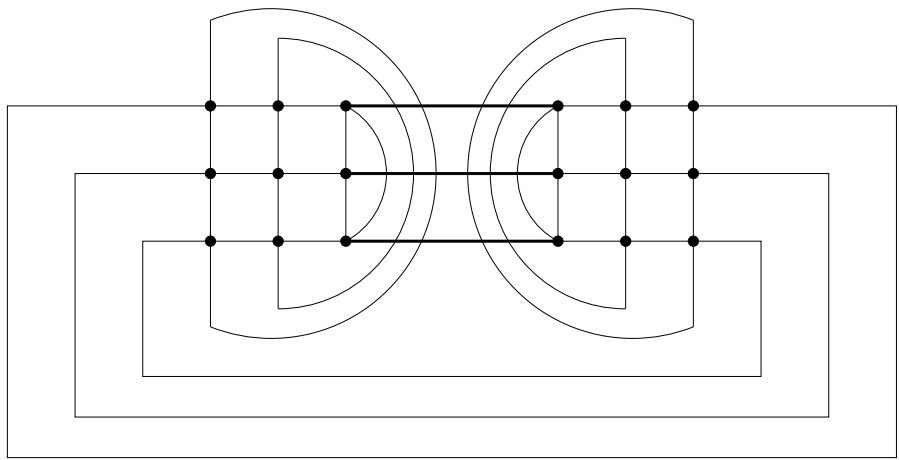}}         
	\caption{A drawing of $C_3\ \Box\ C_6$ with $14$
          crossings, where the thick edges are the edges of one
          particular $E(j)$. 
This is easily generalized to obtain, for every
          even integer $n \ge 2$, a (not crossing-minimal) drawing of $C_3\ \Box\ C_n$ with
          $3n-4$ crossings with the following property: there exists a
          $j\in\{0,1,2,\ldots,n-1\}$ such that  each crossing
          involves an edge in $E(j)$.}
\label{fig:figc3n}
\end{figure}
\end{center}

It is
known that $\ucr(C_3\ \Box\ C_n) = n$ for every $n \ge
3$~\cite{c3n}.  
In 
Figure~\ref{fig:figc3n} we depict how to produce a (not crossing-minimal) drawing of
$C_3\ \Box \ C_n$ with $3n-4$ crossings, for every even integer $n\ge
2$, with the following property: there is a $j\in \{0,1,2\ldots,n-1\}$
such that every crossing involves an edge in $E(j)$ (the edges in
$E(j)$ are the thick edges in Figure~\ref{fig:figc3n}).
Thus, 

\vglue 0.2 cm
{
\leftskip=40pt
\rightskip=40pt
\begin{itemize}
\item[(A)] 
if the edges in $\c3n$ are  are weighted, and there exists
  a $j$ such that the sum of the weights of the edges in $E(j)$ is
  $r$, then such a weighted $\c3n$ has
crossing number at most $r\cdot n$.
\end{itemize}
\par}
\vglue 0.2 cm

For $j=0,1,\ldots,n-1$, denote the weights of the edges in
$E(j)$ by $x^j_{1}, x^j_{2},x^j_3$. 
We have for
$t\le 1$ that
$\proba(x^j_1+x^j_2+x^j_3 > t) =1-t^{3}/3!$. 
Using independence,
$$ 
\proba(\exists j: x^j_1+x^j_2+x^j_3\le t)=
1-(1-t^{3}/6)^n \approx
1-\exp[-nt^{3}/6].
$$

Choosing $t=
6^{1/3} n^{-1/3} \log^{1/3} n
$, this is at least
$1-1/n$.

Now let $s:=\min\{x^j_1 + x^j_2 + x^j_3 \ |
\ j\in\{0,1,\ldots,n-1\}\}$. Thus $s\le t$ with probability at least
$1-1/n$.  In the complementary scenario (which occurs with probability
$<1/n$), $s$ is obviously at most $3$.  Using this observation
together with (A), it follows that
$\EE(\ucr(\c3n))
<\bigl[
(1-1/n) (
(6)^{1/3} n^{-1/3} \log^{1/3} n
) 
+ (1/n)3
\bigr] 
\cdot n 
$
$
< 2 n^{2/3} \log^{1/3} n + 3.
$
\end{proof}


\subsection{Concentration of the expected crossing number and the crossing number
of randomly sparsened graphs}

Continuing in the theme of expected crossing numbers and its interplay
with the decay of crossing numbers, we finally explore the
concentration around the crossing number of a randomly sparsened graph, as well
as the concentration around the expected crossing number of a graph.

Denote $R=R(G,p)$ the random graph obtained from $G$ by randomly and
independently removing edges, each with probability $p$. Using a standard
martingale concentration inequality we show that $\ucr(R)$
is concentrated around its mean.
Let $E(G)=\{e_1,\ldots,e_m\}$, and consider the random variable $\ucr(R)$ as
a Doob's martingale, where the edges are exposed one by one. The length of
the martingale is $|E(G)|$. Removing or adding an edge changes the crossing
number by at most $|E(G)|$. Thus, by the Azuma-Hoeffding's inequality, for
every $\lambda>0$ we have

\begin{equation}\label{cont1}
\proba[|\EE(\ucr(R))-\ucr(R)|> \lambda] \le \exp \biggl[\frac{-\lambda^2}{2|E(G)|^3}\biggr].
\end{equation}

Let $\beta(n)$ be any function tending to infinity. Inequality
\eqref{cont1} shows concentration with radius
$\lambda=\beta(n)|E(G)|^{3/2}$:

\begin{equation}\label{cont2}
\proba[|\EE(\ucr(R))-\ucr(R)|> \beta(n)|E(G)|^{3/2}] \le \exp \biggl[\frac{-\beta(n)^2}{2}\biggr].
\end{equation}

Similary, we can get concentration around the expected crossing number.
Assign to each edge a random variable taking values from $[0,1]$ (which
could be different for each edge),
which provides to each of them a random weight. Formally, it could be a
function $w: E(G) \to \FF$,
where $\FF$ is a collection of random variables taking values from $[0,1]$.
Then $\EE(cr(G,w))$ is the expected crossing number for a given $w$, and
$\ucr(G,w)$ is
a random variable, which is the crossing number of a weighted graph $G$.
As with the random graph $R$ above, resampling the weight of one edge changes the weighted crossing
number by at most $|E(G)|$, and so we obtain:

\begin{equation}\label{cont3}
\proba[|\EE(\ucr(G))-\ucr(G,w)|> \beta(n)|E(G)|^{3/2}] \le \exp \biggl[\frac{-\beta(n)^2}{2}\biggr].
\end{equation}

These inequalities are meaningful only when $G$ is dense enough, i.e.~$|E(G)|\ge n^{5/4}.$
Note that we could have obtained sharper concentration results for
sparse graphs, under the assumption that removing any edge makes the crossing
number drop by $o(|E(G)|)$. 

\section{Concluding remarks}\label{sec:conrem}



Lemma~\ref{lem:workhorse} falls into the realm of light subgraphs. We
recall that the {\em weight} of a subgraph $H$ of a graph $G$ is the sum of the degrees
(in $G$) of its vertices. 
For a class $\gg$  of graphs, define $w(H,
\gg)$ as the smallest
integer $w$ such that each graph $G\in\gg$ which contains a subgraph 
isomorphic to $H$ has a subgraph isomorphic to $H$ of weight at most
$w$.
If $w(H, \gg)$ is finite then $H$ is {\em light} in $\gg$.

Fabrici and Jendrol'~\cite{fabricijendrol} proved that paths (and no
other connected graphs) are light
in the class of $3$-connected planar graphs. Fabrici et
al.~\cite{fabricietal} proved that this remains true even if the
minimum degree is at least $4$, and Mohar~\cite{moharlight} extended
this to $4$-connected planar graphs.

Although some cycles are light in certain families of planar graphs (see
for instance~\cite{jendrol1,jendrol2,madaras,mohar2}), it is easy to
see  that cycles are not light on the
class of planar graphs (consider, for instance, a wheel $W_n$ with $n$
large: each cycle in $W_n$ is either very long or incident with a
large degree vertex). However, as Richter and Thomassen illustrated
in~\cite{rt}, for some applications one does not need the full
lightness condition. A cycle $C$ in a graph is $(\ell,\Delta)$-{\em
  nearly light} if it has length less than $\ell$ and at most one of its
vertices has degree $\Delta$ or greater. Richter and Thomassen proved
that every planar graph has a $(6,11)$-nearly light cycle. This was
later refined in~\cite{lomelisalazar}, where it was shown that if the graphs
under consideration are sufficiently large, then there is a $\Delta>0$
such that a linear
proportion of the face boundaries are $(6,\Delta)$-nearly light.

The concept of $(\ell,\Delta)$-earrings extends the idea of nearly
light cycles: we allow both vertices $u,v$ incident with some edge $e$
to have arbitrarily large degree, and ask for the existence of {\em
  two} cycles that contain $e$, have bounded length, and (other than
$u$ and $v$) bounded degree. The following immediate corollary (since
every $3$-connected graph is obviously 
irreducible) of
Lemma~\ref{lem:earnonplanar}  guarantees the existence of many pairwise
edge-disjoint earrings in $3$-connected planar graphs.

\begin{lemma}\label{lem:3con}
If $G=(V,E)$ is a $3$-connected planar graph, then
$G$ has at least 
$\xcoonenum|E|$ 
pairwise edge-disjoint
$(5000,500)$-earrings.
\end{lemma}

We remark that the linear dependence on $|E|$ in Lemma~\ref{lem:3con}
is clearly best possible, since there cannot be more pairwise
edge-disjoint earrings than edges in a graph.

Finally, it is natural to ask if the $3$-connectedness requirement can
be weakened.  The construction illustrated in Figure~\ref{fig:counter}
answers this in the negative.

\begin{center}
\begin{figure}[htb]
 	\scalebox{0.36}{\input{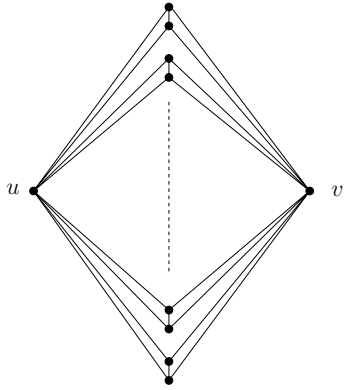}}         
	\caption{The graph $H_n$ obtained by identifying $n$ copies of
          $K_4-e$ on their degree $2$ vertices $u,v$. This family of
          $2$-connected graphs shows that the
          $3$-connectedness condition in Lemma~\ref{lem:3con}
          cannot be
          weakened: for each pair of integers $\ell, \Delta$ there is an
          $n_0:=n_0(\ell,\Delta)$ such that for all $n \ge n_0$, $H_n$
          does not contain any $(\ell,\Delta)$-earring.}
\label{fig:counter}
\end{figure}
\end{center}


It might be argued that the graphs constructed in the proof of
Theorem~\ref{thm:robbust} are somewhat artificial, since many edges
are subdivided a large number of times.  However, these graphs can be
turned into $3$-connected graphs, with equivalent properties, as
follows. Consider the graph $G_2$ in the proof of
Theorem~\ref{thm:robbust}, and some fixed drawing of $G_2$ (for
instance, as in the proof of Theorem~\ref{thm:robbust}, draw the degree
$t-1$ vertices on a circumference, and the branches as the straight
edges joining them).  Let $u_1, u_2, \ldots, u_t$ be the {\em nodes}
(degree $t-1$ vertices) of $G_2$. Thus each branch with endpoints
$u_i, u_j$ can be written as $u_i = u_{i,j}^0, u_{i,j}^1, \ldots,
u_{i,j}^{s-1}, u_{i,j}^s = u_j$ (the same branch, traversing the
vertices in the reverse order, reads $u_j = u_{j,i}^0, u_{j,i}^1,
\ldots, u_{j,i}^{s-1}, u_{j,i}^s = u_i$, so that $u_{i,j}^k =
u_{j,i}^{s-k}$ for $k=0,1,\ldots,s$). Now for each branch $u_{i,j}^0,
u_{i,j}^1, \ldots, u_{i,j}^{s-1}, u_{i,j}^s$, add the edges
$u_{i,j}^k$ and $u_{i,j}^{k+2}$, for $k=0,1,\ldots, s-2$. The
augmented graph is already $2$-connected, but each pair of nodes (that
is, degree $t-1$ vertices) is a $2$-vertex-cut, so we need to
strenghten the connectivity around each node. Consider the node $u_1$,
and suppose for simplicity that the edges $u_1 u_{1,2}^1, u_1
u_{1,3}^1, \ldots, u_1 u_{1,t}^1$ leave $u_1$ in the given (say
clockwise) cyclic order. Then, for each $j=2,3,\ldots, s$, it is
possible to draw an edge from one of $u_{1,j}^1$ and $u_{1,j}^2$ to
one of $u_{1,{j+1}}^1$ and $u_{1,j+1}^2$
without introducing any crossings
(indices are read modulo $s$). By performing this procedure around each
node, we obtain a $3$-connected graph that also witnesses 
Theorem~\ref{thm:robbust}.
The proof is analogous to the proof of Theorem~\ref{thm:robbust}; the only difference  is that instead of requiring
a weak edge of a branch (say between $u_i$ and $u_j$), we need weak triplets
of edges of the form $(u_{i,j}^\ell, u_{i,j}^{\ell+1}),
(u_{i,j}^{\ell-1}, u_{i,j}^{\ell+1}),(u_{i,j}^{\ell}, u_{i,j}^{\ell+2})$,
 where $3\le \ell\le s-3$; we omit the details.

\section*{Acknowledgments}

We thank Bruce Richter and G\'eza T\'oth for very helpful discussions.

\end{document}